\documentclass[letterpaper, 10 pt, conference]{ieeeconf}  

\IEEEoverridecommandlockouts                              
\title{\LARGE \bf
Quadratic Truncated Random Return in Distributional LQR: \\ Positive Definiteness, Density, and Log-Concavity
}
\author{Ruyi Teng, Dan Wang, Wei Chen, and Yulong Gao
  \thanks{R. Teng and Y. Gao are
    with the Department of Electrical and Electronic Engineering,
Imperial College London, United Kingdom. Email: \texttt{ruyi.teng24@imperial.ac.uk, yulong.gao@imperial.ac.uk}. }
\thanks{D. Wang is with the School of Robotics and Automation, Nanjing University, Suzhou, China. {\tt\small danwang@nju.edu.cn}}
\thanks{W. Chen is with the School of Advanced Manufacturing and Robotics, Beijing 100871, China, and also
 with the State Key Laboratory for Turbulence and Complex Systems,
 Peking University, Beijing 100871, China. Email: \texttt{w.chen@pku.edu.cn}. }}

\usepackage{amsmath}
\usepackage{amssymb}
\usepackage{mathrsfs}
\usepackage{mathtools}
\usepackage{bbm}
\usepackage{bm}
\usepackage{amsfonts}
\usepackage{cases}
\usepackage{graphicx}
\usepackage{enumerate}
\usepackage{amsmath}
\usepackage{float}
\usepackage{breakcites}
\usepackage{multirow}
\usepackage{xcolor}
\usepackage{algorithm}
\usepackage{algpseudocode}
\usepackage{graphicx}
\usepackage{subcaption}
\usepackage{epstopdf}
\usepackage{balance}
\usepackage{url}
\usepackage{booktabs}
\usepackage{hyperref}
\usepackage{accents}

\newtheorem{definition}{Definition}[section]
\newtheorem{lemma}{Lemma}[section]

\newtheorem{corollary}{Corollary}[section]
\newtheorem{proposition}{Proposition}[section]
\newtheorem{theorem}{Theorem}[section]
\newtheorem{remark}{Remark}[section]
\newtheorem{example}{Example}[section]
\newtheorem{assumption}{Assumption}[section]

\begin{document}

\maketitle
\thispagestyle{empty}
\pagestyle{empty}

\begin{abstract}

Distributional linear quadratic regulator (LQR) is a new framework that integrates the distributional reinforcement learning and classical LQR, which offers a new way to study the random return instead of the expected cost. Unlike iterative approximation using dynamic programming in the DRL, a closed-form expression for the random return can be exactly characterized in the distributional LQR, which is defined over infinitely many random variables. In recent work~\cite{wang2023policy,wang2023policyL4dc}, it has been shown that this random return can be well approximated by a finite number of random variables, which we call truncated random return. In this paper, we study the truncated random return in the distributional LQR. We show that the truncated random return can be naturally expressed in the quadratic form.  We develop a sufficient condition for the positive definiteness of the block symmetric matrix in the quadratic form and provide the lower and upper bounds on the eigenvalues of this matrix. We further show that in this case, the truncated random return follows a positively weighted non-central chi-square distribution if the random disturbances admits Gaussian, and its cumulative  distribution function is log-concave if the probability density function of the random disturbances is log-concave. 

\end{abstract}

		\section{Introduction}

%

{
Linear Quadratic Regulator  (LQR)~\cite{anderson2007optimal} is one of the most fundamental problems in systems and control. Its importance can be reflected in a large number of practical applications,  e.g., \cite{vandamme2023model,todorov2002optimal}. When the additive disturbance is stochastic, the optimal controller that minimizes the expected quadratic cost is a linear feedback controller, where the optimal feedback gain can be obtained by solving the Riccati  equation.  Inspired by distributional reinforcement learning (DRL)~\cite{bellemare2017distributional}, \emph{distributional LQR} was developed  in  the recent work \cite{wang2023policy,wang2023policyL4dc} to take into account the randomness of total return, beyond its expectation.

The motivation of developing distributional LQR is the benefits of evaluating and exploiting the return distribution in the decision making that  accounts for uncertainty, risk, and multi-modal outcomes~\cite{chow2018risk,wang2023policyL4dc}. 
 In the DRL, it has been shown or demonstrated that the use of return distribution not only improves resilience against unexpected deviations but also improves the efficiency of exploration during learning, guiding policies toward reduced variance and stabilized behavior~\cite{dabney2018distributional}.   
A key result of distributional LQR is the closed-form expression of the random return \cite{wang2023policy,wang2023policyL4dc}, which is intractable for other stochastic systems, including finite Markov decision processes.

The  objective of this paper
is to  study the truncated random return in the distributional LQR.  It was shown in \cite{wang2023policy,wang2023policyL4dc} that truncated random return converges pointwise in distribution to the un-truncated random return while providing improved computational tractability. We show that the truncated random return can be naturally expressed as a quadratic form in random variables, which is a key concept in statistics and applied mathematics~\cite{Mathai1992, varberg1966convergence}. 
To the best of our knowledge, it is the first time to explore the positive definiteness of the block matrix in the quadratic truncated random return. This fundamental property reveals some important results: explicit from of density function and log-concavity.   
Our main contributions are summarized as follows. 
\begin{itemize}
    \item[(1)] We prove that the  symmetric block matrix in the quadratic form of the truncated random return is positive definite under a sufficient condition. 
    \item[(2)] We show that if the random disturbance is zero-mean Gaussian, the truncated random return follows a positively weighted non-central chi-square distribution. 
    \item[(3)] We prove that if the probability density function (PDF) of the random disturbances is log-concave, the cumulative distribution function of the truncated random return is also log-concave. 
\end{itemize}}


        \section{Preliminaries and Problem Statement}\label{Sec:preproblme}

\subsection{Notation}
	We denote by $\mathbb{R}$ the set of real numbers and $\mathbb{N}$ the set of natural numbers. For a symmetric matrix $P$, the notation $P \succ 0$ means that $P$ is positive definite. 
For a matrix $Q\in \mathbb{R}^{n\times n}$, we denote by $\left\| Q\right\|$ the spectral norm. 
To indicate that two random variables $Z_1$ and $Z_2$ are equal in distribution, we use the notation $Z_1 \mathop{=}\limits^{D} Z_2$.
For a random variable $Z$, $\mathbb{E}[Z]$ denotes its expectation. For a real symmetric matrix $H$, $\lambda_{\min}(H)$ denotes its minimal eigenvalue. 

\begin{lemma}\cite[Theorem 2]{feingold1962block} \label{Lemma:blockdomi}
     Let $H=[H_{ij}]$ be a block matrix where $H_{ij}\in \mathbb{R}^{n\times n}$, $i,j=1,\ldots, N$. For each eigenvalue $\lambda$ of $H$, there exists at least one $1\leq i\leq N$ such that 
     \begin{equation}
           (\left\| (H_{ii}-\lambda I)^{-1}\right\|)^{-1}\leq \sum_{j=1,j\neq i}^{N}\left\| H_{ij}\right\|. 
     \end{equation}
\end{lemma}

\begin{definition}\label{Def:blockdominant}
    Let $H=[H_{ij}]$ be a block matrix where $H_{ij}\in \mathbb{R}^{n\times n}$, $i,j=1,\ldots, N$.  If  the diagonal submatrices $H_{ii}$ are nonsingular and 
     \begin{equation}
     \left(\left\| H^{-1}_{ii}\right\|\right)^{-1}\geq \sum_{j=1,j\neq i}^{N}\left\| H_{ij}\right\|, \forall i=1,\ldots, N,\end{equation}
    then $H$ is block diagonally dominant. If the strict inequality is valid, then $H$ is block strictly diagonally dominant.
\end{definition}

\subsection{Canonical discounted LQR}
Consider a discrete time LTI system:
\begin{equation*}
    x_{t+1}=Ax_{t}+Bu_{t}+v_t
\end{equation*}
where  \( x_t \in \mathbb{R}^n \), \( u_t \in \mathbb{R}^p \), and \( v_t \in \mathbb{R}^n \) are the state, control input, and exogenous disturbance, respectively. The exogenous disturbances \( v_t \) are assumed to be independent and identically distributed (i.i.d.) with bounded moments, sampled from an arbitrary distribution \( \mathcal{D} \). 

The discounted LQR is to find a feedback controller  $u_t=\pi(x_t)$ by minimizing
\begin{equation}
J(u) = \mathbb{E} \left[ \sum_{t=0}^{\infty} \gamma^t \left( x_t^\top Q x_t + u_t^\top R u_t \right) \right],
\end{equation}
where \( Q \in \mathbb{R}^{n \times n} \) and \( R  \in \mathbb{R}^{p \times p} \) are symmetric positive definite matrices, \( \gamma \in (0,1) \) is the discount factor. Given an initial state $x_0=x$ and a linear feedback controller  $u_t=K x_t$, the value function is defined as $V_{\pi} (x):=\mathbb{E} \left[ \sum_{t=0}^{\infty} \gamma^t \left( x_t^\top (Q+K^\top R K) x_t\right) \right]$, which fulfills the Bellman equation:
\begin{equation}
V^\pi(x) = x^\top (Q + K^\top R K) x + \gamma \mathbb{E}_{x' = (A + B K)x + v_0} \left[ V^\pi(x') \right],
\end{equation}
where $x'=(A+BK)x+v_0$ is the next random state from $x$. In this case, the solution to the Bellman equation can be explicitly characterized in a quadratic form: 
\begin{equation*}
    V_{\pi}(x)=x^\top P x+ q,
\end{equation*}
where $q$ is a constant and $P$ is the solution to the Lyapunov equation 
\begin{equation}\label{Eq:Lyapunov}
P=Q+K^\top R K+\gamma(A+BK)^\top P(A+BK).
\end{equation}

\subsection{Distributional LQR}

Unlike the traditional discounted LQR which primarily focuses on the expected return, the distributional LQR emphasizes the random return $G^\pi(x)$ with a given feedback controller $u_t = \pi(x_t)$ and an initial state $x_0 = x$:
\begin{align}
    G^\pi(x) = \sum_{t=0}^{\infty} \gamma^t \left( x_t^\top Q x_t + u_t^\top R u_t \right).   
\end{align}

  
It has been shown in \cite{wang2023policy}, if under the linear feedback controller $\pi(x_t)=Kx_t$ that $K$ is stabilizing and satisfies $\|A_{K} \|<1$ with $A_K=A+BK$, the random return $G^{K}(x)$ that fulfills the random-variable Bellman equation 
$$G^{K}(x) \mathop{=}^{D} x^{\top}Q_Kx + \gamma G^{K}(x'),\quad x'=A_K x+ v_0$$
can be exactly characterized as follows:
\begin{align}\label{Eq:Gkx}
      &G^K(x) = x^\top P x + 2 \sum_{k=0}^{\infty} \gamma^{k+1} w_k^\top P A_K^{k+1} x \\  \notag
&+ 
\sum_{k=0}^{\infty} \gamma^{k+1} w_k^\top P w_k + 2 \sum_{k=1}^{\infty} \gamma^{k+1} w_k^\top P \sum_{\tau=0}^{k-1} A_K^{k-\tau} w_\tau, 
\end{align}
where $w_k \sim \mathcal{D} $ are independent from each other for all $k\in\mathbb{N}$ and $P$ is obtained from the Lyapunov equation \eqref{Eq:Lyapunov}. 

Since  $G^K(x)$ involves an infinite number of random variables $w_k$, one can approximate its distribution using a finite number of random variables:  
\begin{align}\label{eq:approxreturn}
    &{G}^{K}_N(x) =  x^{\top} P x  +  2 \sum_{k = 0 }^{N-1} \gamma^{k+1} w_k^{\top} P A_K^{k+1}x \nonumber \\
    &+  \sum_{k = 0 }^{N-1}  \gamma^{k+1}  w_k^{\top} P w_k + 
 2 \sum_{k = 1 }^{N-1} \gamma^{k+1} w_k^{\top} P \sum_{\tau=0}^{k-1} A_K^{k-\tau}w_{\tau},
\end{align}
where $N\in \mathbb{N}$ is the number of random variables.

It has shown in \cite{wang2023policy} that the sequence $\{{G}^{K}_N(x)\}_{N\in \mathbb{N}}$ converges pointwise in distribution to $G^{K}(x)$,  $\forall x\in\mathbb{R}^n$.  More specifically, if the PDF of $w_k$ is bounded,  $\mathbb{E}[w_k^{\top} w_k] \leq \sigma^2$, for all $k\in \mathbb{N}$, and the feedback gain  $K$ satisfies $\left\|A_K\right\| <1$, the sup difference between the CDF $F$ of ${G}^{K}(x)$ and the CDF ${F}_N$ of ${G}_N^{K}(x)$ is bounded by
\begin{align}\label{eq:approximat:bound}
     \sup_{g}|{F}(g)-{F}_N(g) | \leq c_0 \gamma^N, 
\end{align}
where $c_0$ is a constant that depends on the system parameters and the initial state value $x$. 

 In this paper, we will focus on the truncated random return ${G}_N^{K}(x)$ and investigate its fundamental properties, including definite positive quadratic form, density function, and log-concavity. 

     \section{Main Results}\label{Sec:mainresults}

In this section, we begin with the quadratic form of $G^K_N(x)$. 
\begin{proposition}\label{Prop:quadratic}
 Consider the truncated random return $G^K_N(x)$ in \eqref{eq:approxreturn}. Let $Z:=\left[w_0^\top,w_1^\top,w_2^\top,...,w_{N-1}^\top\right]^\top\in \mathbb{R}^{N n\times 1}$.  The  distribution of the truncated random return $G_N^K(x)$ is the same as the distribution of $ S(Z)$
        \begin{equation}\label{Eq:qudratic}
    \begin{aligned}
        G^N_K(x)=S(Z):= Z^\top H_NZ+ 2L_N^\top Z+ c
    \end{aligned}
\end{equation}
where $H_N\in \mathbb{R}^{Nn\times Nn}$ is a symmetric block  matrix with 
\begin{equation}\label{Eq:Hmatrix}
[H_N]_{ij} =
\begin{cases}
\gamma^i P, & \text{if } i = j,\\[6pt]
\gamma^j \bigl(A_K^{j-i}\bigr)^\top P, & \text{if } i < j,\\[6pt]
\gamma^iPA_K^{i-j}, & \text{if } i > j,
\end{cases}
\end{equation}
$L_N \in \mathbb{R}^{Nn\times 1}$ with $[L_N]_i=\gamma^{i}PA_K^{i}x$,  and $c=x^\top P x$. Here $i,j=1,\dots,N$. 
\end{proposition}
\begin{proof}
The quadratic form in \eqref{Eq:qudratic} follows from the definition of $G^K_N(x)$. 
\end{proof}

\begin{remark}
   The quadratic form in random variables is key in statistics~\cite{Mathai1992, varberg1966convergence}. This form facilitates the characterization of probability distributions~\cite{baldessari1967distribution,zhang2025probability}. Many standard distributions (especially chi-squared related distributions) involve quadratic forms.  In addition, this form also plays an important role in related engineering fields, such as communication~\cite{al2009distribution}.  
\end{remark}

\subsection{Positive definiteness}
Given the quadratic form in Proposition~\ref{Prop:quadratic}, we are interested to explore the positive definiteness of the block symmetric matrix $H$. We first give the following preliminary result. 
\begin{lemma}\label{Lem:AhatK}
    Let $Q_K=Q+K^TRK$ and 
    $\hat{A}_{K}=P^{\frac{1}{2}} A_K P^{-\frac{1}{2}}$. We have 
    \begin{itemize}
    \item[(i)]   $\frac{1}{\gamma}\left(I-P^{-\frac{1}{2}}Q_KP^{-\frac{1}{2}}\right)=
\hat{A}_K^\top\hat{A}_K$;
    \item[(ii)] $\|\hat{A}_K\|=\left[\frac{1}{\gamma}\left(1-\lambda_{\min}\left(P^{-\frac{1}{2}}Q_{K}P^{-\frac{1}{2}}\right)\right)\right]^{\frac{1}{2}}$.
    \end{itemize}
\end{lemma}
\begin{proof}
    Note that the matrix $P$ is the solution to the Lyapunov function $P=Q_K+\gamma A_K^\top PA_K$. Since $Q,R$ are positive definite matrices and $A_K$ is stable, we have $P$ is also positive definite. It follows that 
\begin{eqnarray*}
&&P^{\frac{1}{2}}P^{\frac{1}{2}}=Q_K+\gamma P^{\frac{1}{2}}P^{-\frac{1}{2}}A_k^\top P^{\frac{1}{2}}P^{\frac{1}{2}}A_KP^{-\frac{1}{2}}P^{\frac{1}{2}}\\
&&\Leftrightarrow \frac{1}{\gamma}\left(I-P^{-\frac{1}{2}}Q_KP^{-\frac{1}{2}}\right) = \hat{A}_K^\top\hat{A}_K. 
 \end{eqnarray*}
 Since $\hat{A}_K^\top\hat{A}_K \succeq 0$, we have $I-P^{-\frac{1}{2}}Q_KP^{-\frac{1}{2}}\succeq 0$, that is, the matrix $P^{-\frac{1}{2}}Q_KP^{-\frac{1}{2}}$ is stable. 
 It further yields that  \begin{eqnarray*}
&& \|\hat{A}_K\|^2=\sigma^2_{\max}(\hat{A}_K)=\lambda_{\max}(\hat{A}_K^T\hat{A}_K) \\
&& = \frac{1}{\gamma}\lambda_{\max}(I-P^{-\frac{1}{2}}Q_KP^{-\frac{1}{2}}) \\
&&=\frac{1}{\gamma}\left(1-\lambda_{\min}(P^{-\frac{1}{2}}Q_KP^{-\frac{1}{2}})\right),
  \end{eqnarray*}
  which proves (ii). 
\end{proof}

  Let 
 \begin{eqnarray}
\bar{\lambda}&=&\left[\frac{1}{\gamma}\left(1-\lambda_{\min}\left(P^{-\frac{1}{2}}Q_{K}P^{-\frac{1}{2}}\right)\right)\right]^{\frac{1}{2}}, \label{Eq:barlambda}\\ 
\phi_i&=&\sum_{j=1}^i \bar{\lambda}^j+\sum_{j=1}^{N-i}(\gamma \bar{\lambda})^j. \label{Eq:phi} 
 \end{eqnarray}  Before proving the positive definiteness of $H$, we make the following assumption on $\phi_i$.
\begin{assumption}\label{Ass:phi}
Assume that $\phi_i < 1$ for all $i=1,\ldots, N$.
\end{assumption}

\begin{theorem}\label{The:posdef}
Consider the block matrix $H$ in \eqref{Eq:Hmatrix}.   Under Assumption~\ref{Ass:phi}, the following results hold:
 \begin{itemize}
     \item [(i)] $H$ is positive definite, i.e., $H\succ 0$;
     \item [(ii)]  the eigenvalues of $H$ are bounded by  
         \begin{equation*}
         \underset{i=1,\dots,N}{\min}\gamma^i \|P\| (1-\phi_i) \leq \lambda(H) \leq \underset{i=1,\dots,N}{\max} \gamma^i \|P\| (1+\phi_i).
         \end{equation*}
 \end{itemize}
\end{theorem}
\begin{proof}
First, we observe that  
$$H=\Psi \hat{H} \Psi$$ where 
$\hat{H} $ are defined in \eqref{Eq:HASigma} and  
$\Psi={\rm diag}\{P^{\frac{1}{2}},\ldots, P^{\frac{1}{2}}\}$. 
\begin{figure*}
 \begin{equation}\label{Eq:HASigma}
    \hat{H} =\begin{bmatrix}
\gamma I & \gamma^2 P^{-\frac{1}{2}}A_K^\top P^{\frac{1}{2}} & \gamma^3 P^{-\frac{1}{2}}(A_K^2)^\top P^{\frac{1}{2}} & \gamma^4 P^{-\frac{1}{2}}(A_K^3)^\top P^{\frac{1}{2}} & \cdots & \gamma^N P^{-\frac{1}{2}}(A_K^{N-1})^\top P^{\frac{1}{2}} \\
\gamma^2 P^{\frac{1}{2}} A_K P^{-\frac{1}{2}} & \gamma^2 I & \gamma^3 P^{-\frac{1}{2}}A_K^\top P^{\frac{1}{2}} & \gamma^4 P^{-\frac{1}{2}} (A_K^2)^\top P^{\frac{1}{2}} & \cdots & \gamma^N P^{-\frac{1}{2}} (A_K^{N-2})^\top P^{\frac{1}{2}} \\
\gamma^3 P^{\frac{1}{2}} A_K^2P^{-\frac{1}{2}} & \gamma^3 P^{\frac{1}{2}} A_K P^{-\frac{1}{2}} & \gamma^3 I & \gamma^4 P^{-\frac{1}{2}}  A_K^\top P^{\frac{1}{2}}   & \cdots & \gamma^N P^{-\frac{1}{2}} (A_K^{N-3})^\top P^{\frac{1}{2}}  \\
\vdots & \vdots & \vdots & \vdots & \ddots & \vdots\\
\gamma^N P^{\frac{1}{2}}  A_K^{N-1}P^{-\frac{1}{2}} & \gamma^N P^{\frac{1}{2}}  A_K^{N-2} P^{-\frac{1}{2}} & \gamma^N P^{\frac{1}{2}}  A_K^{N-3} P^{-\frac{1}{2}} & \gamma^N P^{\frac{1}{2}}  A_K^{N-4}P^{-\frac{1}{2}} &\cdots& \gamma^N I
\end{bmatrix} 
    \end{equation}
    \hrulefill
        \end{figure*} More specifically, we can write $\hat{H}=[\hat{H}_{ij}]$ where \begin{equation}
\hat{H}_{ij} =
\begin{cases}
\gamma^i I, & \text{if } i = j,\\[6pt]
\gamma^j P^{-\frac{1}{2}} \bigl(A_K^{j-i}\bigr)^\top P^{\frac{1}{2}}, & \text{if } i < j,\\[6pt]
\gamma^i P^{\frac{1}{2}} A_K^{i-j}P^{-\frac{1}{2}}, & \text{if } i > j.
\end{cases}
\end{equation}
Recall  $\hat{A}_{K}=P^{\frac{1}{2}} A_K P^{-\frac{1}{2}}$, which implies that $\hat{A}^j_{K}=P^{\frac{1}{2}} A^j_K P^{-\frac{1}{2}}$. Then, $\hat{H}_{ij}$ can be equivalently written as \begin{equation}\label{Eq:Hhatmatrix}
\hat{H}_{ij} =
\begin{cases}
\gamma^i I, & \text{if } i = j,\\[6pt]
\gamma^j (\hat{A}_{K}^{j-i})^\top, & \text{if } i < j,\\[6pt]
\gamma^i \hat{A}_{K}^{j-i}, & \text{if } i > j.
\end{cases}
\end{equation}
From Definition~\ref{Def:blockdominant},  $\hat{H}$ is block strictly diagonally dominant if $\hat{H}_{ii}$ are nonsingular and 
     \begin{equation}\label{Eq:Hhatdomi}
      \left(\left\| \hat{H}^{-1}_{ii}\right\|\right)^{-1} > \sum_{j=1,j\neq i}^{N}\left\| \hat{H}_{ij}\right\|, \forall i=1,\ldots, N.\end{equation}
Here, $\hat{H}_{ii}=\gamma^i I$, which is naturally nonsingular. From \eqref{Eq:Hhatmatrix}, the inequality  \eqref{Eq:Hhatdomi} can be rewritten as 
\begin{equation}\label{Eq:Hhatdomi2}
 \gamma^i > \gamma^i \sum_{j=1}^{i-1}\left\| \hat{A}\right\|^{i-j} \!+\gamma^i \sum_{j=i+1}^{N}\gamma^{j-i}\left\|\hat{A}\right\|^{j-i}, \forall i=1,\ldots, N.
      \end{equation}
According to Lemma~\ref{Lem:AhatK}, $\left\| \hat{A}\right\|=\bar{\lambda}$ as defined in \eqref{Eq:barlambda}. Then, a sufficient condition for the block strictly diagonal dominance of $\hat{H}$ is 
$\phi_i<1$ for all $i=1,\ldots, N$, where $\phi_i$ is defined in \eqref{Eq:phi}. It follows from Lemma~\ref{Lemma:blockdomi} that for each  eigenvalue $\lambda$ of $\hat{H}$, there exists at least one $1\leq i\leq N$ such that 
     \begin{eqnarray*}
         && |\lambda-\gamma^i|\leq   \gamma^i \phi_i \Leftrightarrow  \gamma^i (1-\phi_i) \leq \lambda\leq \gamma^i (1+\phi_i). 
 \end{eqnarray*}
Note that $\hat{H}^{\top}=\hat{H}$.  Thus,  if $\phi_i<1$ for all $i=1,\ldots,N$,  all the eigenvalues of $\hat{H}$ is positive, which implies that $\hat{H}\succ 0$. Since $H=\Psi \hat{H}\Psi$ and $\Psi \succ 0$, we have that $H\succ 0$. It directly yields the bounds of eigenvalues of $H$ in (ii). 
\end{proof}

Theorem~\ref{The:posdef}  implies that the positive definite property of the block matrix $H$ relies on Assumption~\ref{Ass:phi}, which may restrict the system parameters. The following proposition shows that  Assumption~\ref{Ass:phi} is not required to establish the positive definiteness of $H$ for scalar systems. 
\begin{proposition}
    If the system is a scalar system, 
    i.e., $n=1$, then $H$ in \eqref{Eq:Hmatrix} is positive definite for any $N$. 
\end{proposition}
\begin{proof}
For a scalar system, $A_K$ and $P$ are scalars with $|A_K|<1$ and $P>0$. The positiveness of $H$ directly follows from that its principal minor of order $i$ are positive, i.e.,
   $$ \Delta_i=P^i \gamma^{\frac{i(i+1)}{2}}(1-\gamma A_k^2)^{i-1}>0, \ \forall i=1,\ldots,N,$$
   which completes the proof. 
\end{proof}

\begin{remark}
In \cite{wang2023policy}, $\|A_K\|<1$ is sufficient to ensure good  approximation property of the truncated random return $G^K_N(x)$ where the error bound exponentially decays with respect to $N$.  This condition has to be sharper, as shown in Lemma~\ref{Lem:AhatK} and Assumption~\ref{Ass:phi}, for proving the positive definiteness of $H$, since we need to enforce $\|\hat{A}_K\|=\|P^{\frac{1}{2}}A_KP^{-\frac{1}{2}}\|<1$.  
\end{remark}

\begin{example}\label{Exam:exam1}
    We consider an idealized data center cooling example \cite{dean2020sample,9691800} with the dynamics $x_{t+1}=Ax_t+Bu_t+v_t$, where 
    \[
A = \begin{bmatrix}
1.01 & 0.01 & 0 \\
0.01 & 1.01 & 0.01 \\
0    & 0.01 & 1.01
\end{bmatrix}, \quad
B = \begin{bmatrix}
1 & 0 & 0 \\
0 & 1 & 0 \\
0 & 0 & 1
\end{bmatrix}.
\]
We select $Q=I$ and $R=I$.
The controller is given by
\[
K = -0.015
\begin{bmatrix}
56.19 & 0.7692 & 0.0027 \\
0.7692 & 56.20 & 0.7692 \\
0.0027 & 0.7692 & 56.19
\end{bmatrix},
\]
and $\gamma$ is set to 0.6.  
According to \eqref{Eq:barlambda},  the parameter $\bar{\lambda}=0.1692$. 
We choose different $N$ and compute the corresponding $\phi_i$ in \eqref{Eq:phi}. As shown in Table~\ref{Table:example1}, $\max_{i=1,\ldots,N} \phi_i$ is less than $1$, which implies that Assumption~\ref{Ass:phi} holds for the chosen $N$. We further compute the minimal and maximal eigenvalues of the matrix $H$, both of which are positive, indicating the positive definiteness of $H$. In addition, we further provide the lower and upper bounds in Theorem~\ref{The:posdef}(ii), validating the soundness of these bounds, in comparison with the minimal and maximal eigenvalues of the matrix $H$. 
\end{example}

\begin{table}[htbp]
  \centering
  \caption{The values of $\phi_i$ and eigenvalues of $H$}
  \label{tab:example}
  \begin{tabular}{cccccc}
    \toprule
    $N$   &$\max\limits_{i} \phi_i$ & $\lambda_{\min}(H)$    & $\lambda_{\max}(H)$   & $LB$  & $UB$\\
    \midrule
    2 & 0.2708 & 0.5916 &  1.0848 & 0.3043 &1.3583 \\
    3 & 0.2994 & 0.3547 &  1.0857 & 0.3043 &1.3583 \\5&
    0.3146 & 0.1277 & 1.0857 & 0.1094 &1.3595  \\
    10 & 0.3167 & 0.0099 & 1.0857&0.0085&1.3596 \\
    15 & 0.3167 & 0.0008 & 1.0857&0.0007 &1.3596
 \\
    20 & 0.3167 & 
$6.003 \times 10^{-5}$
 & 1.0857&$5.145 \times 10^{-5}$& 1.3596\\
    25 & 0.3167 & $4.668 \times 10^{-6}$ & 1.0857&$4.001 \times 10^{-6}$& 1.3596\\
    \bottomrule
  \end{tabular}\label{Table:example1}
\end{table}

\subsection{Probability density function}
Thanks to the positive definite quadratic form, we will show that when the disturbances $w_k\sim \mathcal{N}(0,\Sigma)$, we can characterize the density function of the truncated random return $G^K_N(x)$. 
 
\begin{proposition}\label{The:density}
Consider the truncated random return $G^K_N(x)$ in \eqref{eq:approxreturn} and let $Z:=\left[w_0^\top,w_1^\top,w_2^\top,...,w_{N-1}^\top\right]^\top\in \mathbb{R}^{N n\times 1}$.  Assume that  $w_k\sim \mathcal{N}(0,\Sigma)$. Then, the distribution of  $G_{N}^{K}(x) $ (or $S(Z)$ in \eqref{Eq:qudratic}) is a positively weighted non-central chi-square distribution with $Nn$ degrees of freedom
    \begin{equation}
        \begin{aligned}
            G_{N}^{K}(x)=S(Z)&=\sum_{i=1}^{Nn}\lambda_i\left( Y_i+\eta_i \right)^2+x^\top Px-\eta^\top \Lambda \eta\\&+d^\top (\Sigma_Z^{\frac{1}{2}} H \Sigma_Z^{\frac{1}{2}}-H)d,
        \end{aligned}
    \end{equation}
where $\Sigma_Z=I_N\otimes \Sigma$, $\Sigma_Z^{\frac{1}{2}}=I_N \otimes \Sigma^{\frac{1}{2}}$, $Y\in \mathbb{R}^{Nn\times1}$ with $Y_i\sim
\mathcal{N}(0,1)$. Here $\lambda_i>0$, $i=1,\dots, Nn$, $\Lambda = {\rm diag}\{\lambda_i\}\in \mathbb{R}^{Nn\times Nn}$,  $\eta \in \mathbb{R}^{Nn\times1}$, and $d\in \mathbb{R}^{Nn\times1}$ depend on $H_N$, $L_N$, and $c$ in \eqref{Eq:qudratic}. 
\end{proposition}
\begin{proof}
   Recall \eqref{Eq:qudratic}, we have
    \begin{equation*}
        \begin{aligned}
        S(Z)&= Z^\top H_N Z+L_N^\top Z+c=\left( Z+d\right)^\top H_N(Z+d) +c' \\
        &= F^\top H_N F+c',
        \end{aligned}
    \end{equation*}
    where $d=\frac{1}{2}H_N^{-1}L_N$, $c'=c-d^\top H_N d$, $F\in \mathbb{R}^{Nn\times1}$ and $F\sim
\mathcal{N}(d,\Sigma_Z)$. Then $F$ can be be normalized by subtracting $d$ and left multiply $\Sigma_Z^{-\frac{1}{2}}$, which yields 
\begin{equation*}
    \begin{aligned}
        S(Z)&= \left[\Sigma_Z^{\frac{1}{2}}\left( W\!+\!d\right)\right]^\top \! H_N \left[\Sigma_Z^{\frac{1}{2}}\left( W\!+\!d\right)\right]+c'\\
        &= W^\top \Sigma_Z^{\frac{1}{2}} H \Sigma_Z^{\frac{1}{2}}W +2d^\top \Sigma_Z^{\frac{1}{2}} H \Sigma_Z^{\frac{1}{2}}W +c'',
    \end{aligned}
\end{equation*} 
where $W=\Sigma_Z^{-\frac{1}{2}}\left( F-d\right)$, $W \sim \mathcal{N}(\mathbf{0},I_{Nn})$, and $c''=d^\top \Sigma_Z^{\frac{1}{2}} H \Sigma_Z^{\frac{1}{2}}d+c'$.
Let $\Gamma=\Sigma_Z^{\frac{1}{2}} H_N \Sigma_Z^{\frac{1}{2}}$, $\Phi
^\top =d^\top \Sigma_Z^{\frac{1}{2}} H_N \Sigma_Z^{\frac{1}{2}}$. Choose the orthogonal matrix $U$ such that $\Gamma= U\Lambda U^\top$ with $\Lambda$ being a diagonal matrix. Then, we have
\begin{equation*}
    \begin{aligned}
        S(Z)&=W^\top \Gamma W+ 2 \Phi
        ^\top W+c'= W^\top U \Lambda U^\top W \\
        & +2\Phi
        ^\top W+c'' =Y^\top \Lambda Y+2\Phi^\top UY+c''
    \end{aligned}
\end{equation*}
where $Y\sim \mathcal{N}(0,I_{Nn})$. Let $\zeta^\top=\Phi^\top U$, then
\begin{equation*}
    \begin{aligned}
        S(Z)&= Y^\top \Lambda Y+2\Phi^\top UY+c''= Y^\top \Lambda Y+2\zeta^\top Y+c''\\&= (Y+\eta)^\top \Lambda (Y+\eta)+c'''\\&=\sum_{i=1}^{Nn}\lambda_i\left( Y_i+\eta_i \right)^2+c''-\eta^\top \Lambda \eta\\
        &=\sum_{i=1}^{Nn}\lambda_i\left( Y_i+\eta_i \right)^2+d^\top \Sigma_Z^{\frac{1}{2}} H_N \Sigma_Z^{\frac{1}{2}}d+c'- \eta^\top \Lambda \eta \\
        &=\sum_{i=1}^{Nn}\lambda_i\left( Y_i+\eta_i \right)^2+d^\top (\Sigma_Z^{\frac{1}{2}} H_N \Sigma_Z^{\frac{1}{2}}-H_N)d-\eta^\top\Lambda \eta\\&+x^\top Px ,
    \end{aligned}
\end{equation*}
where $\eta=-\Lambda^{-1}\zeta$, $c'''=c''-\eta^\top \Lambda \eta$, which yields a positively weighted non-central chi-square distribution.
\end{proof}

An immediate result is the explicit PDF of $G_N^K(x)$ in the following corollary.  
\begin{corollary} \cite{ruben1962probability,Gardini02102022}
Consider the truncated random return $G^K_N(x)$ in \eqref{eq:approxreturn} and assume that  $w_k\sim \mathcal{N}(0,\Sigma)$. The PDF of $G_N^K$ can be explicitly written in the form 
\[
f_G(g) 
= \sum_{k=0}^{\infty} 
  a_k  \frac{q^{\alpha + k - 1} e^{-\tfrac{q}{2\beta}}}
        {(2\beta)^{\alpha + k} \Gamma(\alpha + k)}
=
\sum_{k=0}^{\infty}
  a_k \, f_\Omega
  \!\bigl(g; \alpha + k, 2\beta\bigr),
\]
$f_\Omega\!\bigl(g; \alpha + k, 2\beta\bigr)$ is the PDF of the gamma distribution of a random variable with $\alpha+k$ shape parameter evaluated at $g$.
The coefficients are defined recursively:
\begin{align*}
a_{0} &= e^{-\frac{\eta}{2}}\prod_{i=1}^{Nn}\biggl(\frac{\beta}{\lambda_i}\biggr)^{\tfrac12}, 
\ a_{k} = (2k)^{k-1}\,\sum_{l=0}^{k-1} b_{k-l}\,a_{l}, 
\quad k \ge 1,
\\
b_{k} &= k\,\beta \sum_{i=1}^{Nn} \eta_i \,\lambda_i^{c_i} 
\;+\;\sum_{i=1}^{Nn} c_i^k, 
\quad k \ge 1, \\
\eta&=\sum_{i=1}^{Nn} \eta_i^2,\quad \alpha=\frac{Nn}{2},\quad c_i=1-\frac{\beta}{\lambda_i}.
\end{align*}
 Here $\beta$ is an arbitrary constant satisfying $|1-\frac{\beta}{\lambda_i} |<1$ to guarantee the $b_k$'s convergence and $\lambda_i$'s are defined in Proposition~\ref{The:density}. 
\end{corollary}
\begin{example}
Let us continue to consider the same system in Example~\ref{Exam:exam1}. We assume that 
the i.i.d random disturbances $w_k$ are zero-mean Gaussian random vector with covariance matrix 
$
\Sigma=\begin{bmatrix}
    1 &0.5 &0.3\\
    0.5& 2& 0.4\\
    0.3 & 0.4 &2
\end{bmatrix}.
$
Let $N=20$ and $x=[3 \ 3 \ 3]^\top$. 
In this case, the PDF of $G^K_N(x)$ is shown Fig.~\ref{fig:gaussianpdf}, consistent with the shape of positively weighted non-central chi-square distribution.
\end{example}
\begin{figure}[H]
  \centering  \includegraphics[width=0.4\textwidth]{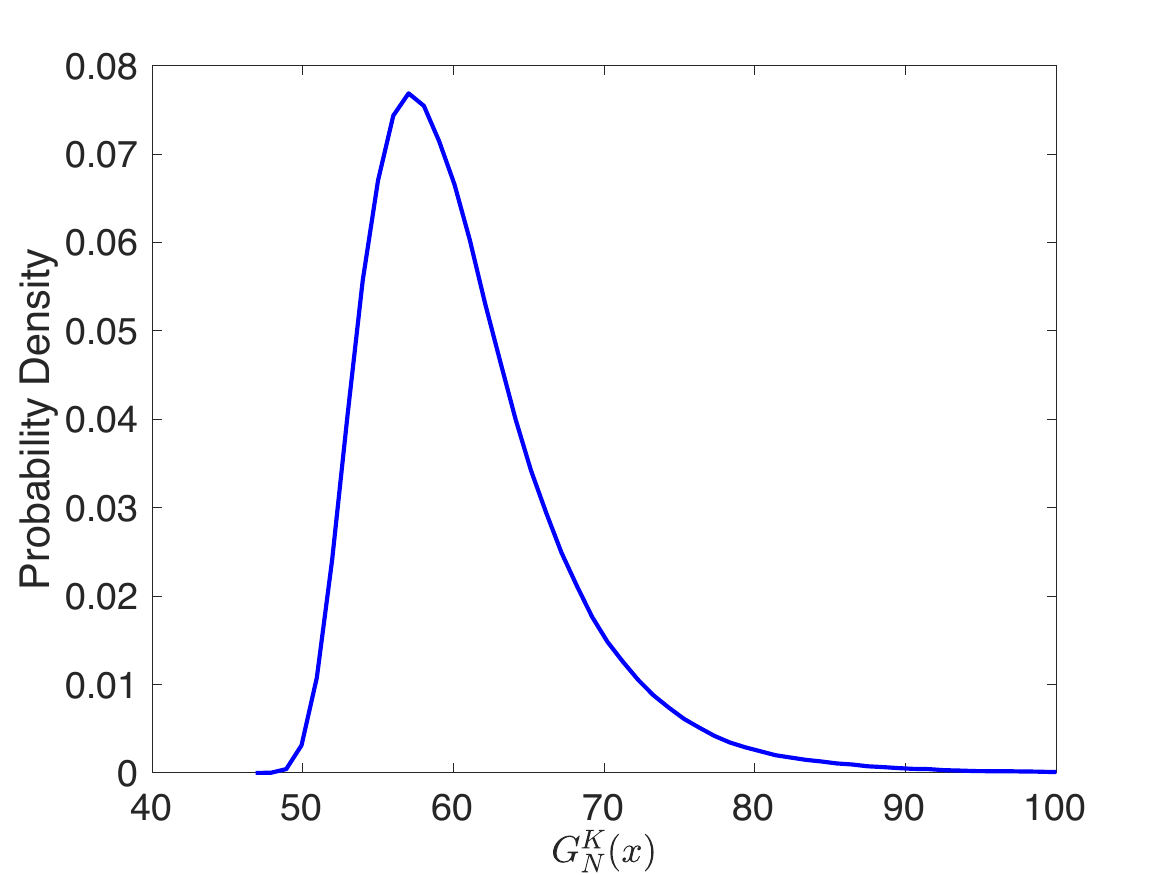}
  \caption{Probability density values of $G_N^K(x)$ where the random disturbance is zero-mean Gaussian.}
  \label{fig:gaussianpdf}
\end{figure}
\begin{remark}
    The positively weighted non-central
chi-square distribution is a special non-central generalized chi-square distribution, which governs the the performance index of the finite-horizon LQR problem~\cite{qian2011complete}.
This kind of distribution plays a significant role in advanced statistical theory~\cite{liu2009new} and applications, e.g., 
  portfolio theory~\cite{guo2024applications}.  
\end{remark}

\subsection{Log-concavity of CDF}

In the light of the positive definite quadratic form of $S(Z)$, we have the following result. 
\begin{lemma}\label{Coro:convexlevelset}
Consider the $S(Z)$ in \eqref{Eq:qudratic} and define the sub-level set of this mapping $S$ at the level $g$:
\begin{equation}\label{Eq:levelset}
\mathcal{L}_g = \{ z \in \mathbb{R}^{Nn} \mid S(z) \leq g \}.
\end{equation} 
The set $\mathcal{L}_g$ is convex and compact for any $g\in \mathbb{R}$.
\end{lemma}
\begin{proof}
It directly follows from the positive definite quadratic form of $S(Z)$ which is convex in $Z$. 
\end{proof}

\begin{proposition}
  \label{The:logconcave}
   Consider the truncated random return $G^K_N(x)$ in \eqref{eq:approxreturn} and let $Z:=\left[w_0^\top,w_1^\top,w_2^\top,...,w_{N-1}^\top\right]^\top\in \mathbb{R}^{N n\times 1}$.  Suppose that the PDF $f_Z(z)$ of $Z$  is log-concave. Then, the CDF ${F}^{K}_{x,N}$ of  the truncated random return $G^K_{N}(x)$ is log-concave.
\end{proposition}
\begin{proof}
From Lemma~\ref{Coro:convexlevelset}, the sub-level set  $\mathcal{L}(g)$  in \eqref{Eq:levelset} of the mapping $S$ is convex and compact for any $g\in \mathbb{R}$. The CDF of $G$ can be defined using the set $\mathcal{L}(g)$:
\begin{equation*}
    {F}_N(g) = \mathbb{P}[G^N_K(x) \leq g] = \int_{\mathcal{L}(g)} f_Z(z) \, dz.
\end{equation*}
The Pr\'ekopa's theorem has shown that, 
   if a function $ f: \mathbb{R}^n \times \mathbb{R}^m \to \mathbb{R} $ is log-concave, then $h(x) = \int f(x, y) \, dy$
is a log-concave function of \( x \).
To apply the Pr\'ekopa's Theorem, we construct an auxiliary function $f(z,g) \in \mathbb{R}^{N}\times \mathbb{R} $: 
\begin{equation*}
f(z,g)=
    \begin{cases}
        f_Z(z), & S(z)\leq g, \\
        0 ,& S(z)>g.
    \end{cases}
\end{equation*}
Therefore, if we can prove that $f(z,g)$ is log-concave on $(z,g)\in \mathbb{R}^{Nn} \times \mathbb{R}$, we can apply the Pr\'ekopa's Theorem to show that ${F}_N(g)$ is log-concave on $\mathbb{R}$.
Since $f(z,g)\equiv0$ when $S(z)>g$, We consider two scenarios by selecting distinct $z_1$, $z_2$, $g_1$, $g_2$:
\begin{enumerate}
  \item \textbf{Case 1:} $f(z_1,g_1)=0$ or $f(z_2,g_2)=0$.  
  The following trivially holds for any $\theta\in[0,1]$: 
  \begin{equation*}
  \begin{aligned}
&f\bigl[(1-\theta)(z_1,g_1) +\theta(z_2,g_2)\bigr]
\\&\ge 0 = 
[f(z_1,g_1)]^{1-\theta}\cdot[f(z_2,g_2)]^{\theta}.
\end{aligned}
\end{equation*}
  \item \textbf{Case 2:} $F(z_1,g_1)\neq 0$, $f(z_2,g_2)\neq 0$. Let $\bigl(z_\theta,g_\theta\bigr) = \bigl((1-\theta)z_1+\theta z_2,\,(1-\theta)g_1 + \theta g_2\bigr)$. Since $S(Z)$ is convex, 
  \begin{equation*}
  \begin{aligned}
      S(z_\theta)&=
S\bigl((1-\theta)z_1 + \theta z_2\bigr)\le
(1-\theta)S(z_1)+ \theta S(z_2)
\;\\&\le
(1-\theta) g_1 + \theta g_2=
g_\theta.
\end{aligned}
  \end{equation*}
As a result, $S(z_\theta)\leq g_\theta$. 
For $f(z_\theta, g_\theta)$, we have
\begin{equation*}
\begin{aligned}
f\bigl((1-\theta)(z_1,g_1) + \theta(z_2,g_2)\bigr)=f(z_\theta, g_\theta )=
f_Z(z_\theta).
\end{aligned}
\end{equation*}
Since $f_Z(z)$ is log-concave on $\mathbb{R}^{Nn}$, we have that for any $\theta \in (0,1)$, $f_Z\bigl(z_\theta)
\ge
\bigl(f_Z(z_1)\bigr)^{1-\theta}\bigl(f_Z(z_2)\bigr)^{\theta}$.
Recalling that $f(z_1,g_1)=f_Z(z_1)$, $f(z_2,g_2)=f_Z(z_2)$ when $f\neq 0$,
\begin{equation*}
\begin{aligned}
&f\bigl[(1-\theta)(z_1,g_1) \,+\,\theta(z_2,g_2)\bigr]=f(z_\theta, g_\theta )\\&=f_Z(z_\theta)\geq \bigl(f_Z(z_1)\bigr)^{\,1-\theta}\,\bigl(f_Z(z_2)\bigr)^{\,\theta}\\&=
\bigl(f(z_1,g_1)\bigr)^{1-\theta}\,\bigl(f(z_2,g_2)\bigr)^\theta.
\end{aligned}
\end{equation*}
\end{enumerate}
To this end, we show that $f(z,g)$ is log-concave. According to the Pr\'ekopa's theorem, we further conclude that the CDF ${F}_N(g)$ is log-concave.
\end{proof}

\begin{remark}
    The log-concavity is an important property in statistical inference and optimization. In particular, the log-concavity of CDF is useful in reliability theory~\cite{bagnoli2006log}, while the log-concavity of PDF implies uni-modality of the density function and allows for efficient MCMC sampling~\cite{dwivedi2019log}.
\end{remark}


     \begin{example}
We verify our proof of log-concavity of CDF on
a discretized linear model for vehicle
steering, which is borrowed from \cite{aastrom2021feedback,kishida2022risk}. The system matrices are $A = [1 \ 0.2; 0 \ 1]$ and $ B= [0.06; 0.20]$. We select $Q=10I$ and $R=1$. 
We consider five different random disturbances $w_k$ whose PDFs are  log-concave. We choose $x=[0;0]$, $N=20$, $\gamma = 0.6$, and $K=[-2.11 \ -2.56 ]$. 
We plot the second derivative of the CDF $F_N(g)$ of the random return $G^K_N(x)$, shown in Fig.~\ref{fig:fivenoise}. We see that for all these different log-concave input disturbances,  the second derivatives of $\log(F_N(g))$ are negative, validating its concavity. 
 \begin{figure}[H]
	\centering
	\includegraphics[width=0.38\textwidth]{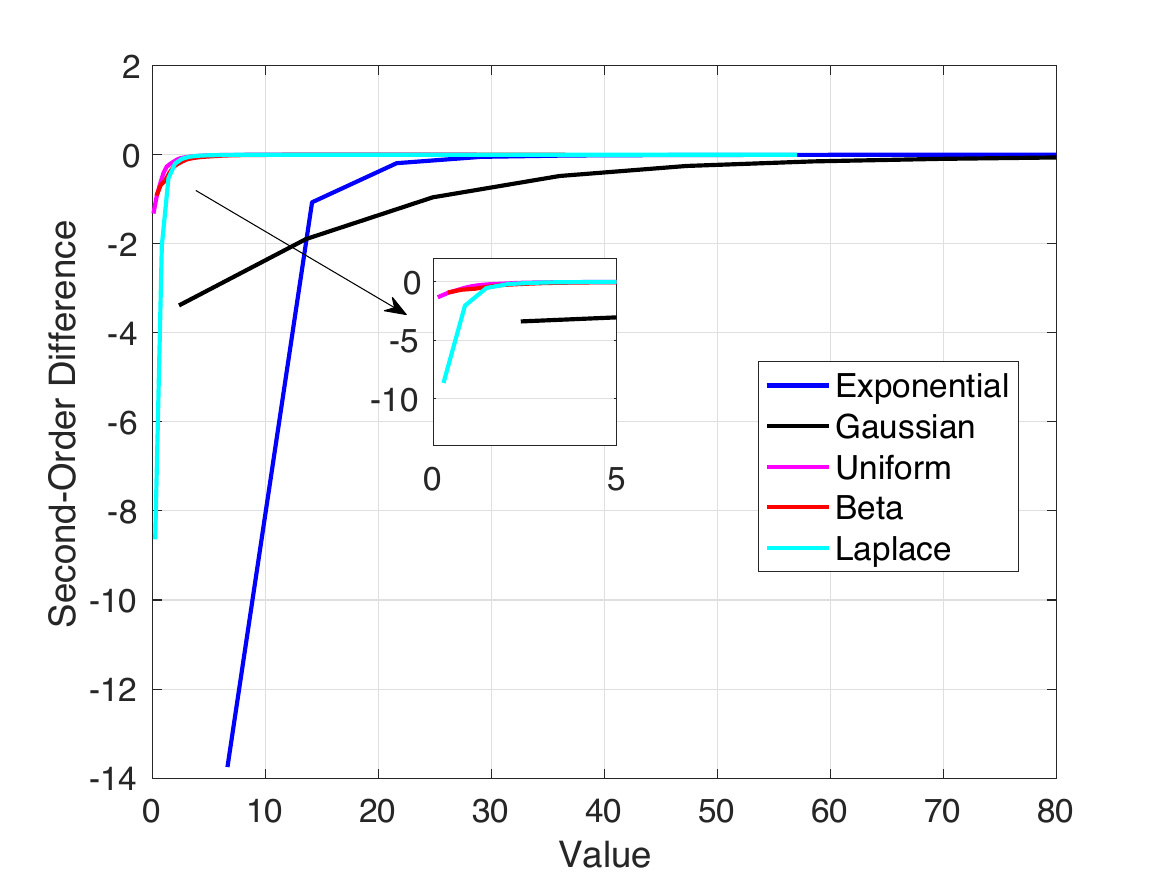}
	\caption{Second derivative of $\operatorname{log}(F_N(g))$ under five different random disturbances whose PDFs are log-concave}
	\label{fig:fivenoise}
\end{figure}
\end{example}

      \section{Conclusion}\label{Sec:conclusion}

In this paper, we have studied the truncated random return in the distributional LQR. We showed that the truncated random return can be naturally expressed as a quadratic form in random variable.  We proved the positive definiteness of the block symmetric matrix in the quadratic form. We further showed that the truncated random return follows a positively weighted non-central chi-square distribution if the random disturbances are Gaussian, and its cumulative distribution function is log-concave if the probability density distribution of the random disturbances is log-concave. 

There are several interesting future directions to be explored. 
The first one is whether we can develop a necessary and sufficient condition for proving the positive definiteness of the block matrix in the quadratic form. The second one is the use of these properties for policy improvement in the risk-averse control. In addition, thanks to the importance of quadratic form in random variables in statistics and other related fields of engineering, we are interested in the role of random return in decision making from a statistical perspective.

\bibliographystyle{IEEEtran}
\bibliography{reference}     

\end{document}